\documentclass[12pt]{amsart}
\usepackage{verbatim,amssymb}
\newtheorem{thm}{Theorem}[section]
\newtheorem{prop}[thm]{Proposition}
\newtheorem{lem}[thm]{Lemma}

\newtheorem{cor}[thm]{Corollary}

\newtheorem*{thmA}{Theorem A}
\newtheorem*{thmB}{Theorem B}
\newtheorem*{thmC}{Theorem C}

\theoremstyle{definition}
\newtheorem{dfn}[thm]{Definition}

\def\Ac{\mathcal{A}} \def\Bc{\mathcal{B}} 
\def\C{\mathbb{C}}  \def\Fc{\mathcal{F}}
 \def\al{\alpha}
  \def\Mc{\mathcal{M}}
    \def\Kc{\mathcal{K}}
\newcommand{\lam}{\mathcal{L}}
\newcommand{\lc}{\mathcal{LC}}

\newcommand{\pr}{\mathrm{Pr}}
\newcommand{\I}{\mathrm{I}}
\newcommand{\kr}{\mathrm{KR}}
\newcommand{\g}{\mathbf{g}}

\def\R{\mathbb{R}}

\def\wt{\widetilde}

\def\Uc{\mathcal{U}} \def\Vc{\mathcal{V}} 

\def\Z{\mathbb{Z}}  

\renewcommand\emptyset{\varnothing}
\newcommand{\sm}{\setminus}

\def\eps{\varepsilon}
\def\La{\Lambda}
\def\al{\alpha}
\def\be{\beta}
\def\da{\delta}
\def\ga{\gamma}
\def\Ga{\Gamma}
\def\Si{\Sigma}
\def\ta{\theta}
\def\om{\omega}
\def\la{\lambda}
\def\si{\sigma}
\def\vp{\varphi}
\def\ol{\overline}

\def\ar{\mathrm{Ar}}
\def\Re{\mathrm{Re}}
\def\cu{\mathrm{CU}}
\def\uc{\mathbb{S}^1}
\def\bd{\mathrm{Bd}}

\def\le{\leqslant}
\def\ge{\geqslant}
\def\0{\emptyset}
\def\disk{\mathbb{D}}
\def\phd{\mathrm{PHD}}

\def\ope{\ol{\phd}^e_3}

\def\ch{\mathrm{CH}}

\begin{document}
\date{May 24, 2013}
\title[The Main Cubioid]
{The Main Cubioid}

\author[A.~Blokh]{Alexander~Blokh}

\thanks{The first and the third named authors were partially
supported by NSF grant DMS--1201450}

\author[L.~Oversteegen]{Lex Oversteegen}

\thanks{The second named author was partially  supported
by NSF grant DMS-0906316}

\author[R.~Ptacek]{Ross~Ptacek}

\author[V.~Timorin]{Vladlen~Timorin}

\thanks{The fourth named author was partially supported by
the Dynasty Foundation fellowship, the Simons-IUM fellowship, RFBR grants
11-01-00654-a, 12-01-33020, and AG Laboratory NRU-HSE, MESRF grant ag. 11.G34.31.0023}

\address[Alexander~Blokh, Lex~Oversteegen and Ross~Ptacek]
{Department of Mathematics\\ University of Alabama at Birmingham\\
Birmingham, AL 35294}

\address[Vladlen~Timorin]
{Faculty of Mathematics\\
Laboratory of Algebraic Geometry and its Applications\\
Higher School of Economics\\
Vavilova St. 7, 112312 Moscow, Russia }

\address[Vladlen~Timorin]
{Independent University of Moscow\\
Bolshoy Vlasyevskiy Pereulok 11, 119002 Moscow, Russia}

\email[Alexander~Blokh]{ablokh@math.uab.edu}
\email[Lex~Oversteegen]{overstee@math.uab.edu}
\email[Ross~Ptacek]{rptacek@uab.edu}
\email[Vladlen~Timorin]{vtimorin@hse.ru}




\begin{abstract}
We discuss different analogs of the main cardioid in the parameter
space of cubic polynomials, and establish relationships between
them.
\end{abstract}

\maketitle

\section{Introduction}
For a complex polynomial $f$, let $J(f)$ be its \emph{Julia set}
and $K(f)$ be its \emph{filled Julia set}. By \emph{classes} of
degree $d$ polynomials, we mean \emph{affine conjugacy classes}.
Denote by $[f]$ the class of a polynomial $f$.

The degree $d$ \emph{connectedness locus} $\Mc_d$ is the set of
classes of degree $d$ polynomials $f$ with connected $K(f)$
(equivalently, $[f]\in\Mc_d$ if all critical points of $f$ belong to $K(f)$). Thus,
the set $\Mc_2$ is identified with the \emph{Mandelbrot set}
consisting of all complex parameters $c$ such that the orbit of $0$
under the polynomial $z^2+c$ is bounded. The central part of the
Mandelbrot set, the \emph{Principal Hyperbolic Domain} $\phd_2$,
consists of all parameter values $c$ such that the polynomial
$z^2+c$ is hyperbolic, and the set $K(z^2+c)$ is a Jordan disk (a
polynomial of any degree is said to be {\em hyperbolic} if the
orbits of all its critical points converge to attracting cycles).
The boundary of $\phd_2$ is called the \emph{Main Cardioid}.

In degree $d$, the Principal Hyperbolic Domain $\phd_d$ of $\Mc_d$
is the hyperbolic component of $\Mc_d$ consisting of classes $[f]$
such that $K(f)$ is a Jordan disk. Equivalently, the class $[f]$ of
a degree $d$ polynomial $f$ belongs to $\phd_d$ if all critical
points of $f$ are in the immediate attracting basin of the same
attracting (or super-attracting) fixed point. Theorem A lists
necessary conditions for $[f]$ to belong to $\ol{\phd}_d$ (a point
$b\in B$ of a topological space $B$ is called a \emph{cutpoint} of
$B$ if $B\sm \{b\}$ is disconnected).

\begin{dfn}[Main Cubioid]\label{d:maincu} The Main Cubioid is the set
$\cu$ of classes of cubic polynomials $f$ with connected $J(f)$ such
that:
\begin{enumerate}
\item
the polynomial $f$ has at least one non-repelling fixed point,
\item
it has no repelling periodic cutpoints of $J(f)$, and
\item all non-repelling periodic points but perhaps
one fixed point have multiplier 1.
\end{enumerate}
\end{dfn}

\begin{thmA}
Let $f$ be a polynomial, whose class belongs to $\ol{\phd}_d$. Then
$f$ has a fixed non-repelling point and no repelling periodic
cutpoints of the Julia set of $f$. Moreover, all non-repelling
periodic points but perhaps one fixed point have multiplier 1. Thus,
$\ol{\phd}_3\subset\cu$.
\end{thmA}

If $[f]\in \ol{\phd}_d$ then $f$ cannot have two attracting periodic
points as otherwise any small perturbation of $f$ will have two
periodic attracting points while there exist polynomials with
classes from $\phd_d$ (and hence with only one periodic attracting
point) arbitrarily close to $f$. A part of Theorem A extends this
observation to non-repelling periodic points.

The name of $\cu$ suggests similarity with the Main Cardioid with
one difference: the quadratic analog of the Main
Cubioid is the \emph{closure} of the Principal Hyperbolic Domain
$\phd_2$ rather than the boundary of it. Observe that, by
definition, if $J(f)$ is disconnected, then $[f]\not\in\cu$. Observe
also that by the Fatou-Shishikura inequality \cite{fat20, shi87}, a
cubic polynomial $f$ has at most two non-repelling cycles.

As in \cite{BOPT-QL},
we now define the \emph{extended} closure $\ol\phd_3^e$ of the
principal hyperbolic domain. To do so we first need to define a
holomorphic motion. Let $\Lambda$ be a Riemann surface, and
$Z\subset\C$ any (!) subset. A {\em holomorphic motion} of the set
$Z$ is a map $\mu:Z\times \Lambda\to\C$ with the following
properties:
\begin{itemize}
  \item for every $z\in Z$, the map $\mu(z,\cdot):z\times \Lambda\to\C$ is
  holomorphic;
  \item for $z\ne z'$ and every $\nu\in\Lambda$, we have $\mu(z,\nu)\ne \mu(z',\nu)$;
  \item there is a point $\nu_0$ such that $\mu(z,\nu_0)=z$ for all $z\in Z$.
\end{itemize}
A crucial result about holomorphic motions is the
\emph{$\lambda$-lemma} of Ma\~n\'e, Sad and Sullivan \cite{MSS}:
{\em a holomorphic motion of a set $Z$ extends to a unique
holomorphic motion of the closure $\ol Z$; moreover, this extension
is a continuous function in two variables}. Suppose that for each
$\nu\in \La$ a map $h_\nu:Z\to \C$ is given. A holomorphic motion
$\mu:Z\times\Lambda\to\C$ is called \emph{equivariant $($with respect
to the family of maps $h_\nu$$)$} if for every $\nu\in\Lambda$ and
every $z\in Z$ with $h_{\nu_0}(z)\in Z$ we have
$h_\nu(\mu(z,\nu))=\mu(h_{\nu_0}(z),\nu)$.

Let $\Fc_\la$ be the space of all cubic polynomials of the form
$$
f_{\la,b}(z)=\la z+bz^2+z^3,\quad b\in \C.
$$
A polynomial $f\in \Fc_\la$ is called \emph{stable with respect to
$0$} if its Julia set admits an equivariant holomorphic motion over
some neighborhood of $f$ in $\Fc_\la$. The \emph{($\la$-)stable} set
$\mathcal{S}_\la\subset \C$ is the set of all $b\in \C$ such that
$f_{\la, b}$ is stable with respect to $0$. A \emph{($\la$-)stable
domain} is a component $\La$ of $\mathcal{S}_\la$. Clearly, for any
$b_1, b_2$ in a $\la$-stable component $\La$ there is an equivariant
holomorphic motion from $J(f_{\la, b_1})$ to $J(f_{\la, b_2})$. A
polynomial $g$ is said to be \emph{stable} if $g\in [f]$ where $f\in
\Fc_\la$ is stable with respect to $0$.

\begin{dfn}\label{d:exte}
The set $\ol\phd_3^e$ is the union of $\ol\phd_3$ and all
$\la$-domains of stability $\La$ with $|\la|\le 1$ such that classes
of polynomials from $\bd(\La)$ belong to $\ol\phd_3$.
\end{dfn}

We conjecture that $\ol{\phd}_3=\cu=\ol\phd_3^e$. To support this
conjecture, we prove Theorem B. Let $\lc$ be the set of classes of
all polynomials with locally connected Julia set.

\begin{thmB}
We have $\ol\phd^e_3\subset \cu$ and
$\lc\cap \cu=\lc\cap \ol\phd^e_3$.
\end{thmB}

To state Theorem C we need to combine \emph{rational laminations}
\cite{kiwi97} and laminations understood as \emph{equivalence
relations} on $\uc$ (see \cite{bl02, bco11}). We discuss the notion
of a lamination in great detail in Section~\ref{s:thmC}; here we
only give necessary definitions. Denote by $\disk\subset\C$ the open
unit disk of radius 1 centered at the origin, and by $\uc$ the
boundary of $\disk$.

\begin{dfn}[Laminations]\label{d:lam}
An equivalence relation $\sim$ on the unit circle $\uc$ is called a
\emph{lamination} if either $\uc$ is one $\sim$-class (such
laminations are called \emph{degenerate}), 
or the following holds:

\noindent (E1) the graph of $\sim$ is a closed subset in $\uc \times
\uc$;

\noindent (E2) if $t_1\sim t_2\in \uc$ and $t_3\sim t_4\in \uc$, but
$t_2\not \sim t_3$, then the open straight line segments in $\C$
with endpoints $t_1, t_2$ and $t_3, t_4$ are disjoint;

\noindent (E3) each equivalence class of $\sim$ is totally
disconnected.
\end{dfn}

A lamination $\sim$ admits a \emph{canonical extension onto $\C$}:
its classes are either convex hulls of classes of
$\sim$, or points which do not belong to such convex hulls. By
Moore's Theorem the space $\C/\sim$ is homeomorphic to
$\C$. The quotient map $p_\sim:\uc\to \uc/\sim$ extends to the plane
with the only non-trivial point-preimages (\emph{fibers}) being the
convex hulls of $\sim$-classes. From now on we will always consider
such extensions of the quotient map.
We write $\sigma_d:\uc\to\uc$ for the map $z\mapsto z^d$.

\begin{dfn}[Laminations and dynamics]\label{d:si-inv-lam}
A lamination $\sim$ is called ($\si_d$-){\em invariant} if:

\noindent (D1) $\sim$ is {\em forward invariant}: for a $\sim$-class $\g$,
the set $\si_d(\g)$ is a $\sim$-class;

\noindent (D2) for any $\sim$-class $\g$, the map $\si_d: \g\to
\si_d(\g)$ extends to $\uc$ as an orientation preserving covering map
such that $\g$ is the full preimage of $\si_d(\g)$ under this covering
map.
\end{dfn}

For a $\si_d$-invariant lamination $\sim$ consider the
\emph{topological Julia set} $\uc/\sim=J_\sim$ and the
\emph{topological polynomial} $f_\sim:J_\sim\to J_\sim$ induced by
$\si_d$. One can extend $f_\sim$ to a branched-covering map
$f_\sim:\C\to \C$ of degree $d$ called a \emph{topological
polynomial} too.
The map $p_\sim$ semi-conjugates $\si_d$ with $f_\sim$, at least
on the unit circle and all leaves of $\sim$.
Unlike complex polynomials, topological polynomials
can have periodic critical points in their topological Julia sets.
The complement $K_\sim$ of the unique unbounded component
$U_\infty(J_\sim)$ of $\C\sm J_\sim$ is called the \emph{filled
topological Julia set}. For $a$, $b\in\uc$, let $\ol{ab}$ be the
{\em chord} with endpoints $a$ and $b$. If $A\subset \uc$ is closed,
boundary chords of the convex hull $\ch(A)$ of $A$ are called
\emph{edges} of $\ch(A)$.

\begin{dfn}[Leaves and gaps]\label{d:lea}
If $A$ is a $\sim$-class, call an edge $\ol{ab}$ of $\bd(\ch(A))$ a
\emph{leaf}. All points of $\uc$ are also called
(\emph{degenerate\emph{)} leaves}. The family $\lam_\sim$ of all
leaves of $\sim$ is called the \emph{geometric lamination {\rm
(}geo-lamination{\rm )} generated by $\sim$}. Let $\lam^+_\sim$ be
the union of all leaves of $\lam_\sim$. The closure of a non-empty
component of $\disk\sm \lam^+_\sim$ is called a \emph{gap} of
$\sim$. Leaves and gaps of $\lam_\sim$ are called
\emph{$\lam_\sim$-sets}; a leaf which is not an edge of a finite gap
is called \emph{independent}. If $G$ is a gap or leaf, we call the
set $G'=\uc\cap G$ the \emph{basis of $G$}.
\end{dfn}

Extend $\si_d$ (keeping the notation) linearly over all
\emph{individual chords} in $\ol{\disk}$ (e.g., over leaves of
$\lam_\sim$); even though the extended $\si_d$ is not well-defined
on the entire disk, it is well-defined on $\lam^+_\sim$. A gap or
leaf $U$ is said to be \emph{{\rm(}pre{\rm)}periodic} if
$\si_d^{m+k}(U')=\si_d^m(U')$ for some $m\ge 0$, $k>0$. If $m$ above
can be chosen to be $0$, then $U$ is called \emph{periodic}; the
minimal number $k$ above is called the \emph{period} of $U$. If $U$
is (pre)periodic but not periodic then it is called
\emph{preperiodic}. A \emph{Fatou} gap is a gap with infinite basis.
By \cite{kiw02} a Fatou gap $G$ is (pre)periodic under $\si_d$.

\begin{dfn}[Rotational sets and numbers]\label{d:rota}
If $\g$ is a periodic non-degenerate finite $\sim$-class of period
$n$, the map $\si_d^{\circ n}|_{\g}$ is conjugate (by a conjugacy
that preserves the cyclic order) to a rigid rotation $R_\rho$ by a rational
angle $\rho$ on a finite $R_\rho$-invariant subset of $\uc$. The
number $\rho$ is then called the \emph{rotation number of $\g$}. A
periodic Fatou gap $G$ of period $n$ such that
$f^{\circ n}_\sim|_{\bd(p_\sim(G))}$ is conjugate to an
irrational rotation by an angle $\rho$, is called a \emph{Siegel}
gap while $\rho$ is called the \emph{rotation number of $G$}.
Otherwise $f^{\circ n}_\sim|_{\bd(p_\sim(G))}$ is conjugate to a map
$\si_k$ with some $k>1$ and $G$ is called a \emph{Fatou gap of
degree $k$}. Siegel gaps and finite $\sim$-classes with non-zero
rotation number are called \emph{rotational sets}.
\end{dfn}

In Section~\ref{s:thmC} we combine ideas from \cite{kiwi97} and
\cite{bco11} and define the geo-lamination $\lam_f$ associated to
$f$. First we associate to $f$ a lamination $\sim_f$ such that the
monotone map from $J(f)$ onto $J_{\sim_f}$ is the finest monotone
map of $J(f)$ onto a locally connected continuum (a map is
\emph{monotone} if point-preimages --- \emph{fibers} --- are continua).
This is possible by \cite{bco11} (see Theorems~\ref{t:statproje} and
~\ref{t:dynproje} in Section~\ref{s:lamisec}).
To specify $\lam_{\sim_f}$ we
need the following definition.

\begin{dfn}\label{d:cut}
A curve $\Gamma$ in the dynamic plane of $f$ consisting of dynamic
(periodic of period $k$) external rays $R_f(\ta_1)$, $R_f(\ta_2)$
and their common landing point $x$ is called a (periodic of
period $k$) \emph{cut} while $x$ is called the \emph{vertex} of $\Gamma$.
\end{dfn}

Finally, we add to $\lam_{\sim_f}$ leaves and finite gaps
corresponding to periodic cuts of the fibers associated to infinite
classes of $\sim_f$. It was shown in \cite{bco11} that a fiber
contains at most finitely many vertices of periodic cuts. We also
add the corresponding pullbacks of the added leaves and gaps. We
make a distinction between different kinds of gaps of $\lam_f$
depending on what they correspond to in terms of $f$. Thus, to each
polynomial we associate a geo-lamination $\lam_f$ equipped with an
additional structure: for each gap or leaf of $\lam_f$, we know if
it is contained in the convex hull of a $\sim_f$-class or not.

\begin{dfn}\label{d:lampair} A \emph{laminational pair} is a pair
$\{\sim, \lam\}$ where $\lam\supset \lam_\sim$ is a geo-lamination
obtained by adding to $\lam_\sim$ finitely many finite periodic gaps
or leaves inside convex hulls of infinite $\sim$-classes as well as
all their preimages so that $\lam$ is a geo-lamination.
\end{dfn}

Now we can define the Combinatorial Main Cubioid.

\begin{dfn}\label{d:cubioid}
A cubic lamination $\sim$ (and its geo-lamination $\lam_\sim$)
is \emph{cubioidal} if each periodic
non-degenerate leaf of $\sim$ has an attached to it Fatou gap whose
basis is not one $\sim$-class, and
$\sim$ has at most one rotational set. A
laminational pair $\{\sim, \lam\}$ is \emph{cubioidal} if
$\sim$ is cubioidal and
$\lam_\sim=\lam$. The {\em Combinatorial
Main Cubioid} $\cu^c$ is the set of all cubioidal geo-laminations
(they are called
\emph{$\cu$-laminations}).
\end{dfn}


There are two extreme cases for $\{\sim_f, \lam_f\}$. First,
$\sim_f$ may identify no two points. Then $J(f)$ is a Jordan curve,
$\lam_f$ has no leaves, and $[f]\in \cu$. We call such laminational
pair \emph{empty}. Second, $\sim_f$ may identify all points of
$\uc$ while $\lam_f$ contains no leaves. By our
Lemma~\ref{l:emptylam} then again $[f]\in \cu$. We call such
laminational pair \emph{degenerate}. The degenerate and the empty
laminational pairs share the same geo-lamination, are cubioidal, and correspond to
polynomials $f$ with $[f]\in \cu$, yet correspond to two very
different types of dynamics. In all other cases $\lam_f$ includes
some non-degenerate leaves.

\begin{thmC}
If 
$[f]\in\cu$ then
$(\sim_f, \lam_f)$ is a cubioidal
laminational pair.
\end{thmC}

\medskip

{\footnotesize \emph{Notation:} we write $\ol A$ for the closure of
a subset $A$ of a topological space and $\bd(A)$ for the boundary of
$A$; the $n$-th iterate of a map $f$ is denoted by $f^{\circ n}$.
For $A\subset\C$ let $\ch(A)$ be the \emph{convex hull} of $A$ in
$\C$. If it does not cause ambiguity, we will often speak of
cutpoints meaning cutpoints of the appropriate Julia sets.
We will consistently identify \emph{angles}, i.e. elements of $\R/\Z$,
with points of the unit circle $\uc\subset\C$.}

\section{Proof of Theorem A}
We first recall some terminology and notation.

\subsection{Dynamic rays}
Let $f(z)=z^d+a_{d-1}z^{d-1}+\dots +a_0$ be a monic degree $d$
polynomial. The Green function $G_f$ is defined by the formula
$$
G_f(z)=\lim_{n\to\infty}\frac{\log_+|f^{\circ n}(z)|}{d^n},
$$
where $\log_+r$ equals $\log r$ if $r>0$ and $0$ otherwise. This
function is harmonic on the complement of the filled Julia set
$K(f)$ of $f$ and is equal to $0$ on $K(f)$. Define {\em dynamic
rays} as unbounded trajectories of the gradient flow for $G_f$. Let
$V(f)$ be the union of all dynamic rays of $f$. Then $V(f)$ is a
forward-invariant open set, and there is a conformal isomorphism
$\phi_f$ between $V(f)$ and some open subset of the set $\{|z|>1\}$
with the following properties:
$$
\phi_f(f(z))=\phi_f(z)^d,\quad G_f(z)=\log|\phi_f(z)|.
$$
These properties define the map $\phi_f$ almost uniquely: the only
way to change the map $\phi_f$ without violating the two properties
is to post-compose it with multiplication by a $(d-1)$-st root of
unity.

The map $\phi_f$ is called a {\em B\"ottcher coordinate}. It is used
to parameterize dynamic rays of $f$. Every dynamic ray is the
preimage of a straight radial ray $\{re^{2\pi i\theta}\,|\,
r>r_0\}$, $r_0\ge 1$, under the map $\phi_f$ (if $K(f)$ is
disconnected then $r_0>1$ if the dynamic ray contains a pre-critical
point in its closure). We will write $R_f(\theta)$ for this ray, and
call it the \emph{dynamic ray of argument $\theta$}. Arguments of
dynamic rays are {\em angles}, i.e. elements of the group $\R/\Z$.
If $f$ is a degree $d$ polynomial, not necessarily monic, then we
can make it monic by a complex linear change of variables. Thus, it
still makes sense to talk about dynamic rays of $f$.

However, arguments of dynamic rays are not well-defined, since they
depend on the choice of a B\"ottcher coordinate. Hence every time we
consider rays in dynamic planes of different polynomials, we must
resolve the issue of choosing the arguments consistently. For
example, if a sequence $f_n$ of degree $d$ polynomials converges to
a degree $d$ polynomial $f$, then we can choose any B\"ottcher
coordinate for $f$, and then, for $f_n$ sufficiently close to $f$,
choose the B\"ottcher coordinate for $f_n$ that is close to the
chosen B\"ottcher coordinate for $f$.

Suppose that the Julia set $J(f)$ is connected. Consider a
periodic repelling cutpoint $\alpha$ of $f$, and let $r$ be its
minimal period. Then there are finitely many dynamic rays landing at
$\alpha$; we will assume that the choice of their arguments is
fixed, and denote the set of the arguments by $\ar_f(\alpha)$. Every
wedge between consecutive rays landing at $\alpha$ contains exactly
one component of $J(f)\setminus\{\alpha\}$. The dynamic rays landing
at $\alpha$ may form one or more orbits under the map $f^{\circ r}$.
Choose one of the orbits, and let $\theta_0$, $\dots$,
$\theta_{q-1}$ denote the arguments of all rays in this orbit
labeled in the counterclockwise order. Suppose that
$d^r\cdot\theta_i=\theta_{i+p\pmod q}$, where $d$ is the degree of
the polynomial $f$. In this case, we say that $\alpha$ has {\em
combinatorial rotation number} $p/q$; observe that $p, q$ are
coprime except for the case when $p/q=0$ and all rays landing at
$\alpha$ are invariant. Every repelling fixed point has a
well-defined combinatorial rotation number (thus, $p/q$ above does
not depend on the choice of $\theta_0$).

\subsection{Polynomials in $\ol{\phd}_d$}
We now recall Lemma B.1 from \cite{GM} that goes back to Douady and
Hubbard \cite{DH}.

\begin{lem}
 \label{l:rep}
Let $f$ be a polynomial, and $z$ be a repelling periodic point of
$f$. If the ray $R_f(\theta)$ lands at $z$, then, for every
polynomial $g$ sufficiently close to $f$, the ray $R_{g}(\theta)$
lands at a repelling periodic point $w$ close to $z$. Moreover, $w$
depends holomorphically on $g$.
\end{lem}

Choose a polynomial $f$ with $[f]\in\ol{\phd}_d$. For brevity, by a
\emph{cutpoint} we mean a cutpoint of $J(f)$. We want to prove the
following statements:
\begin{enumerate}
\item
the map $f$ has no repelling periodic cutpoints, and
\item
the map $f$ has at most one non-repelling periodic point of
multiplier different from $1$.
\end{enumerate}

Statement (1) follows from Lemma \ref{l:rep}. Indeed, suppose that
$\alpha$ is a repelling periodic cutpoint of $f$. Let $\theta_0$,
$\dots$, $\theta_{r-1}$ be the arguments of dynamic rays landing at
$\alpha$. Since $\alpha$ is a cutpoint, then $r>1$. Lemma
\ref{l:rep} says that, for $g$ sufficiently close to $f$, the
dynamic rays with arguments $\theta_0$, $\dots$, $\theta_{r-1}$ in
the dynamic plane of $g$ land at the same periodic point that is
obtained from $\alpha$ by analytic continuation. We get a
contradiction if we choose $g$ such that $[g]\in\phd_d$.

The second statement is a consequence of the {\em Yoccoz
inequality}, see e.g. \cite{Hu}. It follows from the Yoccoz
inequality that, for any sequence of polynomials $f_n$ with
repelling fixed points $\al_n\to \al$, the fact that $f_n'(\alpha_n)\to
e^{2\pi i\rho}$ implies that the combinatorial rotation numbers of
$f_n$ at $\alpha_n$ converge to $\rho$. We can now prove
Lemma~\ref{l:lim-attr}.

\begin{lem}
\label{l:lim-attr}
  Consider a polynomial $f$, whose class belongs to $\ol{\phd}_d$,
  and a sequence of polynomials $f_n$ converging to $f$, whose
  classes belong to $\phd_d$.
  If $f$ has a non-repelling fixed point $\alpha$, whose multiplier is
  different from 1, then $\alpha$ is the limit of the attracting fixed points of $f_n$.
\end{lem}

\begin{proof}
Let $\al$ be neutral and
 $\alpha_n$ be a fixed point of $f_n$ such that $\alpha_n\to\alpha$
 as $n\to\infty$.
 If $\alpha_n$ are attracting for arbitrarily large $n$, then we are done.
 Otherwise assume that for all large $n$, the points $\alpha_n$ are repelling
 fixed points. Then their combinatorial rotation numbers equal 0.
 By the assumptions, $f'(\alpha)=e^{2\pi i\rho}$, where
 $\rho\not\equiv 0\pmod{2\pi}$. On the other hand, by the Yoccoz
 inequality, the combinatorial rotation numbers of $f_n$ at
 $\alpha_n$ (which are all equal to 0) must converge to $\rho$, a contradiction.
\end{proof}

We can now complete the proof of Theorem A.

\begin{proof}[Proof of Theorem A]
Observe that if $[f]\in \ol{\phd}_d$ then one of the non-repelling
cycles of $f$ must be a fixed point (indeed, as we approximate $f$
with polynomials $g$, whose classes belong to $\phd_d$, the
attracting fixed points of $g$ converge to a non-repelling fixed
point of $f$). Let $[f]\in\ol{\phd}_d$. By way of contradiction,
suppose that $\alpha$ and $\beta$ are two non-repelling periodic
points, whose multipliers are different from one. Replacing $f$ with a
suitable iterate, we may assume that $\alpha$ and $\beta$ are fixed.
At least one of the two points, say, $\alpha$, is not the limit of
the attracting fixed points of polynomials $g$ with $[g]\in\phd_d$
approximating the polynomial $f$. But this is a contradiction with
Lemma \ref{l:lim-attr}.
\end{proof}

\section{Proof of the first part of Theorem B}
\label{s:thmB} By definition, all polynomials from $\ope$ have a
fixed non-repelling point $\al$. Mapping $\al$ to 0 by an affine
transformation, rotating and rescaling if necessary, we reduce
polynomials from $\ope$ to the form
$$
f_{\lambda,b}(z)=\lambda z+bz^2+z^3, |\la|\le 1.
$$
Let $\Fc$ denote the space of all cubic polynomials of the form
$f_{\la,b}$ and $\Fc_{nr}$ denote the subspace of $\Fc$ consisting
of polynomials $f_{\la,b}$ with $|\la|\le 1$.

Fix $\la$ with $|\la|\le 1$. Let $\mathfrak{g}_{\la,b}$
be the Green function for $K(f_{\la,b})$. Let $V_{\la,b}$ be the
union of all unbounded trajectories of the gradient flow generated
by $\mathfrak{g}_{\la,b}$. The B\"ottcher coordinate is an analytic
map $\phi_{\la,b}:V_{\la,b}\to\C$ with $\phi_{\la,b}\circ
f_{\la,b}=\phi_{\la,b}^3$ and $\phi_{\la,b}(z)=z+o(z)$ as
$z\to\infty$.

\begin{thm}[\cite{BH}, Proposition 2]
 \label{T:BH}
Let $\Vc_\la$ be the union of $\{b\}\times V_{\la,b}$ over all
$b\in\C$. This set is open in $\C^2$. The map
$\Phi_\la:\Vc_\la\to\C^2$ given by the formula
$\Phi_\la(b,z)=(b,\phi_{\la,b}(z))$ is an analytic embedding of
$\Vc$ into $\C^2$.
\end{thm}

We will write $R_{\la,b}(\ta)$ for the dynamic ray $R_{f_{\la,b}}(\ta)$.

\subsection{Polynomials with parabolic points and their petals}

Let $g$ be a polynomial of arbitrary degree such that $0$ is a fixed
parabolic point of $g$ of multiplier 1. Suppose that
$g(z)=z+az^{q+1}+o(z^{q+1})$, where $q$ is a positive integer.
Recall from \cite{M} that an \emph{attracting vector} for $g$ is
defined as a vector (=complex number) $v$ such that $av^q$ is a
negative real number, i.e. $v$ and $av^{q+1}$ have opposite
directions. Clearly, there are $q$ straight rays consisting of
attracting vectors that divide the plane of complex numbers into $q$
\emph{repelling sectors}.

Consider a repelling sector $S$. Note that the set
$S^{-q}=\{z\in\C\,|\, z^{-q}\in S\}$ is the complement of the ray
$\{-ta\,|\, t>0\}$ in $\C$. Let $U$ be a sufficiently small disk
around $0$. We will write $F$ for the composition of the function
$w\mapsto w^{-1/q}$ mapping $(S\cap U)^{-q}$ onto $S\cap U$, the
function $g$ mapping $S\cap U$ onto $g(S\cap U)$, and the function
$z\mapsto z^{-q}$ mapping $g(S\cap U)$ to $\C$. We have
$F(w)=w-qa+\alpha(w)$, where $\alpha(w)$ denotes a power series in
$w^{-1/q}$ that converges in a neighborhood of infinity, and whose
free term is zero (note that this function is single valued and
holomorphic on $(S\cap U)^{-q}$). It follows that there exists a positive
real number $r$ with the property $|\alpha(w)|<\frac{|a|}2$ whenever
$|w|>r|a|$. Consider the half-plane $\Pi$ given by the inequality
$\Re(w/a)>r$. Since this inequality implies that $|w|>r|a|$, we have
$F(\Pi)\supset\Pi$, and also that the shortest distance from a point
on the boundary of $\Pi$ to a point on the boundary of $F(\Pi)$ is at
least $(q-\frac 12)|a|$. The preimage $RP$ of the half-plane $\Pi$
under the map $z\mapsto z^{-q}$ from $S$ to $S^{-q}$ is called a
\emph{repelling petal} of $g$.

Every repelling sector includes a repelling petal; thus, our
polynomial $g(z)=z+az^{q+1}+o(z^{q+1})$ has $q$ repelling petals. A
repelling petal $RP$ of $g$ is such that $g(RP)\supset RP$. Let us
discuss the dependence of the repelling petals on parameters.
Suppose that a polynomial $\tilde g(z)=z+\tilde a z^{q+1}+\dots$ is
very close to the polynomial $g(z)$, in particular, the coefficients
$a$ and $\tilde a$ are very close to each other. We will assume that
$a$ is nonzero, and that $\tilde a/a$ is close to 1. Consider
repelling sectors $S$ and $\tilde S$ of $g$ and $\tilde g$,
respectively. Since $a$ and $\tilde a$ are close, we can choose the
repelling sectors $S$ and $\tilde S$ to be close.

The half-planes $\Pi$ and $\tilde\Pi$ associated with polynomials
$g$ and $\tilde g$ in the same way as above (we may choose the same
sufficiently large $r$ for both $g$ and $\tilde g$) are also close in the
Hausdorff metric associated with the spherical metric (although this
is not true for the Euclidean metric). The preimages of $\Pi$ and
$\tilde\Pi$ under the map $z\mapsto z^{-q}$ acting on $S^{-q}$ and
$\tilde S^{-q}$, respectively, are also close. Thus we obtain the
following lemma (cf. the proof of Lemma 5 in \cite{BH}).

\begin{lem}
  \label{l:rp}
  Let $g_t(z)=z+a_t z^{q+1}+o(z^{q+1})$ be a continuous family
  of polynomials, in which $a_t$ never vanishes.
  Then all $q$ repelling petals of $g_t$ can be chosen to vary continuously
  with respect to the parameter.
\end{lem}

\subsection{Stability of rays and their perturbations}
Throughout this subsection, we fix $\lambda$ that is a root of
unity, i.e. $\lambda=\exp(2\pi ip/q)$ for some relatively prime $p$
and $q$. Since $\la$ is fixed, we will skip $\la$ from the notation
$f_b$, $R_b(\theta)$ etc. We discuss conditions that guarantee that
a dynamical ray $R_{b}(\theta)$ landing at $0$ is stable, i.e., for
$b'$ sufficiently close to $b$, the ray $R_{b'}(\theta)$ also lands
at $0$.

\begin{prop}
\label{P:Tpq}
  We have $f_{b}^{\circ q}(z)=z+T_{p/q}(b)z^{q+1}+o(z^{q+1})$,
  where $T_{p/q}(b)$ is a non-zero polynomial in $b$.
\end{prop}

\begin{proof}
By the Petal Theorem \cite[Theorem 6.5.10]{Beardon}, we have
$$
f_{b}^{\circ q}(z)=z+T_{p/q}(b)z^{q+1}+o(z^{q+1}).
$$

It remains to prove that the polynomial $T_{p/q}(b)$ cannot be
identically equal to zero. For any $b$ such that $T_{p/q}(b)=0$, the
polynomial $f_{b}$ has at least two cycles of attracting petals at
$0$. Each of the associated cycles of Fatou domains must contain a
critical point of $f_{b}$. Thus both critical orbits of $f_{b}$
converge to $0$. However, for large $b$, one of the critical points
escapes. Therefore, for such $b$, we have $T_{p/q}(b)\ne 0$.
\end{proof}

\begin{prop}
\label{p:ray0stable}
  Suppose that a dynamic ray $R_{b_*}(\theta)$ with periodic $\theta$
  lands at $0$, and $T_{p/q}(b_*)\ne 0$.
  Then, for all $b$ sufficiently close to $b_*$,
  the ray $R_{b}(\theta)$ lands at $0$.
\end{prop}

\begin{proof}
By Lemma 10.1 of \cite{M}, the ray $R_{b_*}(\theta)$ must be tangent
to some repelling vector of $f_{b_*}^{\circ q}$ at $0$. Let
$RP_{b_*}$ be the corresponding repelling petal of $f_{b_*}^{\circ
q}$. The period of $\theta$ is equal to $q$ by Theorem 18.13 of
\cite{M}. There are two points $z_*$ and $f_{b_*}^{\circ q}(z_*)$ in
$R_{b_*}(\theta)$ that lie in the interior of $RP_{b_*}$. By Lemma
\ref{l:rp} for all $b$ sufficiently close to $b_*$, we can define a
repelling petal $RP_{b}$ of $f_{b}^{\circ q}$ that is close to
$RP_{b_*}$ in the Hausdorff metric.

Let $L_*$ denote the subray of the ray $R_{b_*}(\theta)$ from $z_*$
to infinity. By Theorem \ref{T:BH}, for every $\eps>0$, we can
choose a neighborhood $U$ of $b_*$ such that, for all $b\in U$, the
corresponding piece $L$ of $R_b$ is $\eps$-close to $L_*$ in the
Hausdorff metric. The number $\eps$ can be chosen so that this
implies that $L$ enters the corresponding petal $RP_b$. Dynamics
inside $RP_b$ implies that $R_b(\theta)$ lands at $0$.
\end{proof}

\subsection{The proof of inclusion $\ope\subset \cu$}\label{ss:thmb1}

By Theorem A, we have $\phd_3\subset \cu$. Consider a
two-dimensional domain $\Uc\subset\Fc_\la$ consisting of stable
polynomials such that for all $f\in\bd(\Uc)$, we have
$[f]\in\ol\phd_3$. We need to prove that classes of all polynomials
in $\Uc$ belong to $\cu$. Suppose that $f_*=f_{\la,b_*}\in \Uc$, and
show that then $f_*$ has properties (1)--(3) from
Definition~\ref{d:maincu}.

\medskip

\noindent\textbf{Property (1).} Clearly, $f_*$ has a fixed
non-repelling point $0$, thus property (1) is fulfilled.

\medskip

\noindent\textbf{Property (2).} Let us prove that $f_{*}$ has no
repelling cutpoints. Assume that $f_{*}$ has a repelling periodic
cutpoint $z_{b_*}$. The set $\ar_{f_*}(z_{b_*})$ of arguments of
external rays of $f_{*}$ landing at $z_{b_*}$ consists of at least
two angles. Since all maps in $\Uc$ are quasi-symmetrically
conjugate, it is easy to see (e.g., by Lemma 3.5 \cite{BOPT-QL})
that all maps $f_{\la,b}\in \Uc$ have repelling periodic cutpoints
$z_b$ corresponding to $z_{b_*}$. By Lemma~\ref{l:rep}
$\ar_{f_{\la,b}}(z_b)=\ar_{f_*}(z_{b_*})$.
Suppose that $\{\al,\be\}\subset\ar_{f_*}(z_{b_*})$.

Let $\Lambda$ be the set of all parameter values $b$ with
$f_{\la,b}\in\Uc$, and choose a sequence $b_n\to b'\in
\bd(\Lambda)$. We may assume that $z_{b_n}\to z_{b'}$, where
$z_{b'}$ is a non-attracting periodic point of $f_{\la,b'}$. If both
rays $R_{b'}(\al)$, $R_{b'}(\be)$ land at \emph{repelling} periodic
points, then these landing points must coincide as otherwise by
Lemma~\ref{l:rep} we get a contradiction with the fact that
$R_{b_n}(\al)$, $R_{b_n}(\be)$ land at $z_{b_n}$ and $z_{b_n}\to
z_{b'}$. However, by Theorem A, the map $f_{\la,b'}$ does not have
repelling periodic cutpoints. Hence one of the rays $R_{b'}(\al)$,
$R_{b'}(\be)$ lands at a parabolic periodic point. Clearly, for at
most finitely many parameter values $b'\in \bd(\Lambda)$ the rays
$R_{b'}(\al)$ or $R_{b'}(\be)$ land at a parabolic point distinct
from $0$. Assume that for infinitely many $b'\in \bd(\Lambda)$ the
rays $R_{b'}(\al)$ land at $0$ which is a parabolic fixed point:
$\lambda=\exp(2\pi ip/q)$ for some relatively prime $p$ and $q$.

Let us show that in the above case $T_{p/q}(b')=0$. Indeed,
arbitrarily close to $b'$, there are parameter values $b$, for which
$R_b(\al)$ does not land at $0$. It follows from Proposition
\ref{p:ray0stable} that $T_{p/q}(b')=0$. However, the polynomial
$T_{p/q}$ has only finitely many roots, a contradiction.

\medskip

\noindent\textbf{Property (3).} Suppose that $f_{*}$ has a
non-repelling $n$-periodic point $z_{b_*}\ne 0$ with multiplier not
equal to 1. Since $f_{*}$ is stable, the corresponding periodic
point $z_b\ne 0$ of $f_b$, $b\in \Lambda$, is $n$-periodic and
non-repelling. If $b\to b'\in \bd(\Lambda)$, then $z_b\to z_{b'}$
where $z_{b'}$ is a non-repelling $f^{\circ n}_{b'}$-fixed point.
Consider two cases. First, suppose that $z_{b'}\ne 0$. Then by
Theorem A the multiplier at $z_{b'}$ is $1$.
There are only finitely many values of $b'$, for which this can happen.
Second, suppose that $z_{b'}=0$.
We have $f_b^{\circ n}(z)-z=z(z-z_b)Q_b(z)$ for some polynomial $Q_b$,
whose coefficients are algebraic functions of $b$ that have no poles in $\C$.
We obtain in the limit as $b\to b'$ that
$f_{b'}^{\circ n}(z)-z=z^2Q_{b'}(z)$, hence $0$ is a parabolic fixed point of $f_{b'}$.
We may assume that the multiplier at $0$ is $e^{2\pi ip/q}$. Let us show that then $0$ is a
degenerate parabolic point (i.e., that $T(b')=0$ where $T=T_{p/q}$
is the polynomial introduced in Proposition~\ref{P:Tpq}).

Indeed, $n=mq$ is a multiple of $q$, and as in
Proposition~\ref{P:Tpq} by the Petal Theorem $f_b^{\circ
q}(z)=z+T(b)z^{q+1}+o(z^{q+1})$. It is easy to see by induction that
then for any $k$ we have $f_b^{\circ
kq}(z)=z+kT(b)z^{q+1}+o(z^{q+1})$. On the other hand, as above
$f_b^{\circ mq}(z)-z=z^{q+1}(z-z_b)R_b(z)$ where $R_b(z)$ is a
polynomial of $z$ whose coefficients are algebraic functions of $b$ that
have no poles in $\C$.
Hence in the limit we have $f_{b'}^{\circ mq}(z)-z=z^{q+2}R_{b'}(z)$.
It follows that $mT(b')z^{q+1}+o(z^{q+1})=z^{q+2}R_{b'}(z)$ and hence
$mT(b')+o(1)=zR_{b'}(z)$, which implies that $T(b')=0$ as desired.
Clearly, there are finitely many such values of $b'$. Thus, we
showed that overall there are only finitely many values of $b'$ to
which $b$ may converge, a contradiction with $\bd(\Lambda)$ being
infinite.

\section{Laminations associated to polynomials}\label{s:lamisec}

If $X\subset \C$ is a continuum, let $U_\infty(X)$ be the unbounded
component of $\C\sm X$. If $X=\bd(U_\infty(X))$, we call
$X$ \emph{unshielded}. 
A continuous map $\vp:Y\to Z$ is \emph{monotone} if all fibers are
continua. Let $A$ be a continuum. A monotone onto map $\vp:A\to
Y_{\vp, A}$ with locally connected $Y_{\vp, A}$ is called a
\emph{finest (monotone) map} if for any monotone map $\psi:A\to L$
onto a locally connected continuum $L$ there is a map $h:Y_{\vp,
A}\to L$ with $\psi=h\circ \vp$ (then $h$ is monotone because for
$x\in L$, we have $h^{-1}(x)=\vp(\psi^{-1}(x))$). If $\vp:A\to B$,
$\vp':A\to B'$ are two finest maps, then the map associating points
$\vp(x)\in B$ and $\vp'(x) \in B'$ for every $x\in A$ is a
homeomorphism between $B$ and $B'$. Hence all sets $Y_{\vp, A}$ are
homeomorphic, all finest maps $\vp$ are the same up to a
homeomorphism, and we can talk of \emph{\textbf{the} finest model
$Y_A=Y$ of $A$} and \emph{\textbf{the} finest map $\vp_A=\vp$ of $A$
onto $Y$}. Recall that given a lamination $\sim$ the corresponding
quotient map from $\uc$ to $\uc/\sim$ is denoted by $p_\sim$.

\begin{thm}[Theorem 1 \cite{bco11}]\label{t:statproje}
  Let $Q$ be an unshielded continuum. Then there exist the finest
  map $\vp$ and the finest model $Y$ of $Q$ given by a lamination
  $\sim_Q$ on $\uc$ so that $Y=\uc/\sim_Q$; moreover, $\vp$ can be extended to
  a map $\vp:\C\to \C$ which collapses only those complementary
  domains to $Q$ whose boundaries are collapsed by $\vp$, and is a
  homeomorphism elsewhere in $\C\sm Q$. For $y\in Y$ the fiber
  $\vp^{-1}(y)$ coincides with the topological hull of the union of impressions
  of all external to $Q$ rays with arguments from the set
  $p_{\sim_Q}^{-1}(y)$.
\end{thm}

By a \emph{finest map} we mean any extension  of the finest map of
$Q$ over $\C$.

\begin{dfn}[Critical leaves and gaps]\label{d:crit}
A leaf of a lamination $\sim$ is called \emph{critical} if its
endpoints have the same image. A gap $G$ is said to be
\emph{critical} if $\si_d|_{G\,'}$ is at least $k$-to-$1$ for some $k>1$.
\end{dfn}

Lemma~\ref{l:crepe} is well known; we state it here without a proof.

\begin{lem}\label{l:crepe}
An edge of a periodic gap is either (pre)critical or (pre)periodic.
\end{lem}

Laminations give a full and exact description of the dynamics of
complex polynomials with locally connected Julia set \cite{thu85}.
In general, weaker results can be proven \cite{kiwi97, bco11}.
Namely, in \cite{bco11} Theorem~\ref{t:statproje} is applied to
polynomials with connected Julia set which yields
Theorem~\ref{t:dynproje} (a similar earlier result is due to Kiwi
\cite{kiwi97}).

\begin{thm}[\cite{bco11}, Theorem 2]\label{t:dynproje}
  Let $f$ be a complex polynomial with connected Julia set
  and finest lamination $\sim_f=\sim_{J(f)}$. Then there exists a
  topological polynomial $f_{\sim_f}:\C\to \C$ and a finest map
  $\vp_f:\C\to \C$ which semiconjugates $f$ and $f_{\sim_f}$. If
  $x\in J_{\sim_f}$ corresponds to a finite periodic $\sim_f$-class
  $p^{-1}_{\sim_f}(x)$ then the fiber $\vp_f^{-1}(x)$ is a point. No
  periodic Fatou domain of $f$ of degree greater than $1$ is
  collapsed by $\vp_f$.
\end{thm}

We need the following definition.

\begin{dfn}\label{d:ingacsfib}
Call gaps \emph{finite} or \emph{infinite} if
their bases are finite or infinite; infinite $\sim$-classes
have infinite gaps as their convex hulls (such gaps will be called
\emph{infinite gap-classes}). By \cite{bl02} all such gaps are
(pre)periodic, and periodic infinite gap-classes are Fatou gaps of
degree greater than 1. Call the corresponding fibers
\emph{CS-fiber}. Thus, if $x\in J_{\sim_f}$ corresponds to an
infinite gap-class $p^{-1}_{\sim_f}(x)$ then the $x$-fiber
$\vp^{-1}_f(x)$ is said to be a \emph{CS-fiber}.
\end{dfn}

The following lemma explains the terminology.

\begin{lem}[\cite{BOPT-QL}, Proposition 4.4]\label{l:4.4}
A periodic CS-fiber contains either a Cremer point or a Siegel
point.
\end{lem}

The drawback of using the lamination $\sim_f$ to model the dynamics
of $f$ is that $\sim_f$ may incompletely reflect the properties of
repelling periodic cutpoints of $J(f)$. In a lot of cases this does
not happen. Indeed, if all $\sim_f$-classes are finite then by
Theorem~\ref{t:dynproje} there is a one-to-one correspondence
between repelling and parabolic periodic cutpoints of $f$ and their
preimages on the one hand and the (pre)periodic non-degenerate
classes of $\sim_f$. However in the case when some $\sim_f$-classes
are infinite this may no longer be the case.

E.g., suppose that a cubic polynomial $f$ has a fixed repelling
point $0$ at which $R_f(0)$ and $R_f(\frac12)$ land, and no more
repelling periodic cutpoints. Moreover, suppose that in each
``half-plane'' created by the cut $R_f(0)\cup \{0\}\cup R_f(\frac12)$
there is a fixed Cremer point. Denote these fixed Cremer points by
$a$ and $b$. Then a standard argument (see, e.g., \cite{M03}) shows
that there are two quadratic-like Julia sets in $J(f)$, namely $J_a$
(containing $a$) and $J_b$ (containing $b$). Each of them
corresponds to a quadratic Julia set with a Cremer fixed point,
and by \cite{bo06} the only monotone map of $J_a$ ($J_b$) onto a
locally connected continuum is a collapse to one point. It follows
that the only monotone map of $J$ onto a locally connected continuum
is a collapse to one point. Hence the lamination $\sim_f$ identifies
all points of the circle and misses the fact that $f$ has a fixed
repelling cutpoint $0$.

Let $F$ be a fiber associated with an infinite $\sim_f$-class. We
saw that $F$ may contain periodic repelling points cutting $F$ such
that the corresponding leaves are not included in $\lam_{\sim_f}$.
By Proposition 40 \cite{bco11} there are at most finitely many
repelling or parabolic cutpoints in $F$. To each such point $x$ we
associate the convex hull of the set $\ar_f(x)$. We add the edges of
such convex hulls to $\lam_{\sim_f}$. Then we add to $\lam_{\sim_f}$
the edges of gaps corresponding to preimages of such points.
Finally, we take the limit leaves of this family of leaves and add
them to $\lam_{\sim_f}$. This creates a new geo-lamination $\lam_f$
called the \emph{geo-lamination generated by $f$}. In this way we
combine $\lam_{\sim_f}$ with the \emph{rational lamination} defined
by Kiwi in \cite{kiwi97}. In $\lam_f$ we will distinguish between
Fatou gaps corresponding to non-degenerate Fatou domains of $f$,
infinite gap-classes $G$, and infinite gaps $H$ of $\sim_f$ such
that $H\cap \uc$ is one $\sim_f$-class subdivided by finitely many
finite gaps or leaves and their preimages as in the definition of
$\lam_f$.

\section{Proof of Theorem C}
\label{s:thmC}

Lemma~\ref{l:emptylam} deals with the case when $\lam_f$
has only degenerate leaves.
We will assume that $0$ is a fixed point of $f$.

\begin{lem}\label{l:emptylam}
If all leaves of $\lam_f$ are degenerate then $[f]\in \cu$, and
$0$ is the unique non-repelling periodic point of $f$.
Moreover, if $\sim_f$ consists of one class, then $0$ is a Cremer or
Siegel fixed point.
\end{lem}

\begin{proof}
We may assume that $\sim_f$ consists of one class coinciding with
$\uc$. By definition of $\lam_f$ the map $f$ has no repelling
periodic cutpoints, and by Theorem~\ref{t:dynproje}, the polynomial $f$ has no
attracting or parabolic periodic points. By Lemma~\ref{l:4.4} the
point $0$ is a fixed Cremer or Siegel point. Suppose that there is a
non-repelling periodic point $x\ne 0$ of $f$. Similar to the above
$x$ is also a Cremer or a Siegel periodic point. Then by \cite{Ki}
there exists a repelling periodic point separating $x$ and $0$,
a contradiction.
\end{proof}


\begin{lem}\label{l:lamnotlam}
If $[f]\in \cu$ then $\lam_{\sim_f}=\lam_f$.
\end{lem}

Before we prove Lemma \ref{l:lamnotlam}, we need to recall a
description of quadratic invariant gaps given in \cite{BOPT}. Let
$G$ be a degree 2 invariant gap of some $\si_3$-invariant
lamination. Then there is a unique edge $M$ of $G$ (the \emph{major}
of $G$) separating the circle into two arcs, one of which contains
all vertices of $G$ and is of length at most $\frac23$; $M$ must be
a critical leaf or a periodic leaf. Moreover, all edges of $G$ are
iterated $\si_3$-preimages of $M$.\ Suppose that a quadratic
invariant gap $G$ is a gap of $\sim_f$. By \cite[Theorem
7.7]{BOPT-QL}, if $M=\ol{\ta_1\ta_2}$ is a periodic major of $G$,
then the external rays $R_f(\ta_1)$, $R_f(\ta_2)$ land at the same
point. This implies that if $\ol{\al\be}$ is a (pre)periodic edge of
$G$ then the external rays $R_f(\al)$, $R_f(\be)$ land at the same
point.

\begin{proof}[Proof of Lemma \ref{l:lamnotlam}] Suppose that $[f]\in \cu$ and
$\lam_{\sim_f}\ne \lam_f$. Then $f$ has a periodic CS-fiber $F$. By
Lemma~\ref{l:4.4} there exists a Cremer or Siegel periodic point
$y\in F$. If $F$ is not invariant then $y$ is not fixed
contradicting Definition~\ref{d:maincu}. Thus $F$ is invariant, and
we may assume that $y=0\in F$ is a fixed Cremer or Siegel point. As
above, by \cite{bl02} the corresponding to $F$ invariant Fatou gap
$G$ is of degree greater than 1. If $G$ is of degree 3 then $G=\uc$.
Since $[f]\in \cu$, the map $f$ does not have repelling periodic cutpoints.
Since $G=\uc$, then by Theorem~\ref{t:dynproje} the map $f$ cannot have
parabolic periodic points. Hence by definition in this case all
leaves of $\lam_f=\lam_{\sim_f}$ are degenerate.
Assume that $G$ is of degree 2. Since
$[f]\in \cu$ has no repelling cutpoints, then
$\lam_{\sim_f}\ne \lam_f$ implies that there is a parabolic periodic
cutpoint $x$ of $F$. Since $f$ is cubic, by the Fatou-Shishikura
inequality the union of the orbit of $x$ and the point $0$ is the set of
all non-repelling periodic points of $f$. In particular, there are no
other periodic cutpoints of $F$.

Let $\widetilde X$ be the union of all rays landing at $x$ and
$\{x\}$ itself. Some edges of the convex hull $X$ of $\ar_f(x)$ are
contained inside $G$ (otherwise $x$ would not be a cutpoint of $F$).
Apply the map $\vp_G$ which collapses to points all edges of $G$. It
semiconjugates $\si_3|_G$ to $\si_2$ so that the restriction of
$\lam_f$ onto $G$ induces a $\si_2$-invariant geo-lamination
$\lam^2_f$ which contains, by the above, some periodic leaves. By
Proposition II.6.10b of \cite{thu85}, the lamination $\lam^2_f$ has an invariant
gap $H$ of non-zero rotation number or the leaf
$H=\ol{\frac13\frac23}$. Theorem II.5.3 of \cite{thu85} shows that
if $H$ is a gap then either $H$ is a Siegel gap, or it is a gap with
countably many vertices, or it is a finite gap. However in the first
two cases it follows that the lamination $\lam^2_f$ contains an
isolated critical leaf. On the other hand, the construction of
$\lam^2_f$ implies that all non-degenerate leaves of $\lam^2_f$ are
either (pre)periodic with non-degenerate images, or limits of
(pre)periodic, a contradiction. Thus, either
$H=\ol{\frac13\frac23}$, or $H$ is a finite gap of rational rotation
number.

Consider the convex hull $H_1$ of $\vp_G^{-1}(H')$. Then $H_1$ has
either the same number of vertices as $H$, or twice as many vertices
as $H$ (if vertices of $H$ are $\vp_G$-images of edges of $G$). We
want to prove that there is an $f$-fixed point associated to $H_1$
such that external rays of $f$ whose arguments are vertices of $H_1$
land at that point. Indeed, suppose otherwise. Then by definition of
our laminations we may assume that there are $\si_2$-pullbacks of
$\vp_G(X)$ accumulating on each edge of $H$. Let $\ell=\ol{ab}$ be
an edge of $H$. Then the corresponding $\si_3$-pullbacks of $X$ will
accumulate on the corresponding edge of $\ol{a_1b_1}$ of $H_1$. The
corresponding cuts of $F$ formed by the corresponding pullbacks of
$\widetilde X$ can be chosen so that their ``vertices'' (i.e.,
corresponding pullbacks of $x$) converge to a point $y_\ell$
belonging to the impression of $R_f(a_1)$ and the impression of
$R_f(b_1)$. Thus, impressions of $R_f(a_1)$ and $R_f(b_1)$ are
non-disjoint.

If $H_1$ and $H$ have the same number of vertices, it follows that
the union $K$ of all impressions of angles with arguments which are
vertices of $H_1$ is a continuum. If $H_1$ has twice as many
vertices as $H$, for every vertex $l$ of $H$ there is an edge
$\ell=\ol{uv}$ of $H_1$ such that $\vp_G(\ell)=l$. By \cite[Theorem
7.7]{BOPT-QL}, the external rays $R_f(u)$, $R_f(v)$ land at the same
point. Hence in that case the union $K$ of impressions of angles
which are vertices of $H_1$ is a continuum too. Clearly, $K$ is
invariant and separated from impressions of rays with arguments
which are not vertices of $H_1$ (either by the just discussed
pullbacks of $\widetilde X$, or by the appropriate fibers
approaching $F$). By \cite[Lemma 37]{bco11} then $K$ is a fixed
repelling or parabolic point. Since $[f]\in \cu$, $K$ is parabolic.
Since $H$, and hence $H_1$, are of non-zero rotation number, the
multiplier at $K$ is not one. On the other hand, $x$ is a Cremer or
Siegel point. Thus, $f$ has at least two periodic points of
multiplier not equal to 1, a contradiction with $[f]\in \cu$. This
shows that $\lam_{\sim_f}=\lam_f$.
\end{proof}

\begin{proof}[Proof of Theorem C]
In view of Lemma \ref{l:lamnotlam}, it remains to prove that, for
$[f]\in\cu$, the lamination $\lam_f=\lam_{\sim_f}$ is cubioidal. Let
us prove that $\lam_f$ has at most one rotational set $G$, and $G$
is invariant. Suppose that $G'$ is a finite $\sim_f$-class. Then, by
Theorem~\ref{t:dynproje}, it corresponds to a periodic repelling or
parabolic cutpoint $y(G)=y$ of $J(f)$. Since $[f]\in \cu$, then, by
Definition~\ref{d:maincu}(2), the point $y$ is parabolic and by
Definition~\ref{d:maincu}(3) $y=0$. Hence $\lam_f$ cannot have two
finite rotational classes. Now, if $G$ is a Siegel gap of $\sim_f$
then there must exist a Siegel periodic point $y$ of $f$ inside
$\vp_f^{-1}\circ p_{\sim_f}(G)$; thus, $y=0$. Hence $\lam_f$ has at
most one rotational set $G$, and $G$ is invariant.

Again, let $G$ be a finite rotational $\sim_f$-class. Since $y(G)=y$
is a cutpoint of $J(f)$, by Definition~\ref{d:maincu}(2), the point
$y$ is parabolic. Hence there are parabolic domains attached to $y$.
By Theorem~\ref{t:dynproje} they are not collapsed by $\vp_f$. Hence
along at least one cycle of edges of $G$ such that the period of the
endpoints of these edges is, say, $m$, there are Fatou gaps of
period $m$ attached to $G$ and which do not correspond to one
$\sim_f$-class as required in Definition~\ref{d:cubioid}. By
\cite[Corollary 5.5]{BOPT} this implies that for every periodic leaf
$\ell$ of $\lam_f$ whose endpoints are of period $t$ there exists a
Fatou gap of $\lam_f$ of period $t$ attached to $\ell$. It remains
to prove that such gaps cannot be contained in convex hulls of
$\sim_f$-classes. 

By the above the only hypothetical situation which we need to
consider is as follows: there is a periodic finite gap or leaf $G$
of $\lam_f$ with two cycles of edges on its boundary such that Fatou
gaps which \emph{are not} convex hulls of a single $\sim_f$-class are
attached to one of these cycles of edges while Fatou gaps which
\emph{are} convex hulls of a single $\sim_f$-class are attached to the
other cycle of edges. Denote by $H$ a Fatou gap which is one
$\sim_f$-class attached to an edge of $G$; let $F$ be the
corresponding CS-fiber. By Lemma~\ref{l:4.4} there is a Cremer or
Siegel point $x\in F$. Since $[f]\in \cu$, the point $x=0$ is fixed and so $H$
is invariant. Clearly, the only way it can happen is when
$G=\ol{0\frac12}$, a contradiction since if $\ol{0\frac12}$ is a
leaf of $\lam_f$ then from at least one side it has an attached Fatou
gap which does not coincide with the convex hull of a
$\sim_f$-class as desired (so that $\lam_f$ is a CU-lamination).
\end{proof}

We can partially reverse Theorem C. First we prove
Lemma~\ref{l:norepcut}.

\begin{lem}\label{l:norepcut}
If a cubic polynomial $f$ has no repelling cutpoints then it has a
non-repelling fixed point.
\end{lem}

\begin{proof}
Consider fixed external rays $R_f(0)$ and $R_f(1/2)$. If they land
at the same point $w$ then by the assumptions $w$ is non-repelling
as desired. Suppose that the ray $R_f(0)$ lands at $z$, the ray
$R_f(1/2)$ lands at $y$, and $z\ne y$. By \cite{GM} there exists
either an invariant Fatou domain $U$ or a fixed point $x\in
J(f)\sm\{y, z\}$. In the first case $f$ has either an attracting or
a Siegel fixed point, and we are done. In the second case there are
two possibilities. First, a periodic ray $R$ may land at $x$. By the
assumption about $R_f(0), R_f(\frac12)$ the ray $R$ is not
invariant, hence $x$ is a cutpoint. Since $f$ does not have
repelling periodic cutpoints, $x$ is parabolic and we are done.
Second, suppose that no periodic ray lands at $x$. Then $x$ is a
Cremer fixed point, and we are done.
\end{proof}

\begin{lem}\label{l:cuciscu}
Suppose that $(\sim_f, \lam_f)$ is a cubioidal laminational pair,
$f$ has no repelling periodic cutpoints and at most one periodic attracting
point. Then $[f]\in \cu$.
\end{lem}

\begin{proof}
By Lemma~\ref{l:norepcut} we may assume that $0$ is an $f$-fixed
point, $|f'(0)|\le 1$, and if there is a fixed non-repelling
point with multiplier not equal to $1$ then $f'(0)\ne 1$. By Definition~\ref{d:cubioid}, $\lam_{\sim_f}=\lam_f$. By
Definition~\ref{d:maincu} we need to show that all non-repelling
periodic points of $f$ but perhaps $0$ have multiplier $1$. Assume
the contrary: $f$ has a periodic non-repelling point $x\ne 0$, whose
multiplier is different from 1.

We need an observation concerning any parabolic point $y$ of $f$.
By \cite{kiw02} either there is one cycle of rays
landing at $y$, or there are two cycles of rays landing at $y$. In
the first case inside each wedge at $y$ there is a parabolic Fatou
domain attached to $y$. In the second case \emph{a priori} it may
happen that there is one cycle of Fatou domains attached to $y$
inside one cycle of wedges at $y$, and the other cycle of wedges at
$y$ contains no Fatou domains attached to $y$ inside them. However
since $(\sim_f, \lam_f)$ is cubioidal, it
follows that if there are two cycles of rays (and hence wedges) at
$y$, then there are two cycles of Fatou domains at $y$. Now we can
consider several cases.

(1) Assume that $0$ is attracting. Then there is an invariant Fatou
domain $U$ containing $0$. If $x$ is attracting, Cremer or Siegel
then again by \cite{Ki} there exists a repelling periodic cutpoint,
a contradiction. Assume that $x$ is parabolic. Then the fact that
the multiplier at $x$ is not $1$ implies that $x$ cannot be a
boundary point of $U$. By the above there are two cases. First,
there may be one cycle of rays and one cycle of Fatou domains at
$x$. Clearly, then we can find a point from the orbit of $x$ and a
Fatou domain attached to it which can only be separated from $0$ by
a repelling periodic cutpoint, a contradiction. Second, there may be
two cycles of Fatou domains at $x$. Together with $U$ they will
form \emph{three} cycles of Fatou domains of a cubic polynomial $f$,
a contradiction.

(2) Assume that $0$ is Cremer or Siegel. If $x$ is attracting,
Cremer or Siegel then by \cite{Ki} there exists a repelling periodic
cutpoint separating $0$ and $x$ in $J(f)$, a contradiction. Suppose
that $x$ is parabolic. As in (1), the fact that $f$ is cubic implies
that there is exactly one cycle of Fatou domains at
$x$. However this implies that there will be one of Fatou domains at
one of the points of the orbit of $x$ which can only be separated
from $0$ by a repelling periodic cutpoint, a contradiction.

(3) Assume that $0$ is parabolic. By the above there are two
subcases here. First, assume that there are two cycles of Fatou
domains at $0$. Let $G$ be the convex hull of $\ar_f(0)$. If $G$ is
a gap, then each cycle of Fatou domains at $0$ consists of at least
two domains. If one of them is a cycle of attracting Fatou domains,
then we have at least two attracting periodic points of $f$, a
contradiction. If both are cycles of parabolic domains then clearly
we cannot have a non-repelling periodic point $x\ne 0$. Thus, we may
assume that $G$ is a leaf. Then having two cycles of Fatou domains
at $0$ (actually, each cycle in this case consists of just one Fatou
domain) means having two cycles of Fatou gaps attached to $G$ which
implies that $G=\ol{0\frac12}$. If both Fatou domains at $0$ are
parabolic, we cannot have a non-repelling periodic point $x\ne 0$.
Hence one of the Fatou domains at $0$ is attracting and the other
one is parabolic. However in that case by our choice of $f$ we
should have moved the attracting fixed point to $0$, a
contradiction.

Second, assume that there is one cycle of Fatou domains and one
cycle of rays landing at $0$. Then it is easy to see (similar to the
arguments above) that there must exist a repelling cutpoint
separating one of these Fatou domains at $0$ from a
specifically chosen Fatou domain at one of the points from the
orbit of $x$. In any case, we get a contradiction with the
assumption that $f$ has no repelling periodic cutpoints.
\end{proof}


\section{Proof of the second part of Theorem B}

We need to prove that $\lc\cap \cu=\lc\cap \ol\phd^e_3$ ($\lc$ is the set of classes of polynomials
with locally connected Julia sets). By the first part
of Theorem B $\ol\phd^e_3\subset \cu$. Hence we have to consider
cubic polynomials $f$ such that $[f]\in \cu\sm \ol\phd^e_3$. By
Theorem C, the laminational pair $(\sim_f,\lam_f)$ is cubioidal. We may
assume that $f\in \Fc_{nr}$.

\subsection{Main analytic tools}\label{ss:mt}

According to \cite{BOPT-QL}, there is a well-defined \emph{principal
critical point} $\om_1(f)$ of $f$ that depends holomorphically on
$f$ at least in a small neighborhood of $f$ in $\Fc_{nr}$. If
$\lambda=f'(0)$ is a root of unity, then $\om_1$ is in a parabolic
domain attached to $0$, in particular, the orbit of $\om_1(f)$
converges to $0$.

\begin{thm}[\cite{BOPT-QL}, Theorem B]\label{t:mainql}
If $f\in\Fc_{nr}$ and $[f]\not\in\ol\phd_3^e$ then there are Jordan
domains $U^*$ and $V^*$ such that $f:U^*\to V^*$ is a quadratic-like
map hybrid equivalent to $z^2+c$ with $c\in\ol\phd_2$.
\end{thm}

We will write $J^*$ for the Julia set of the quadratic-like map
$f:U^*\to V^*$, and $K^*$ for the filled Julia set of this map.
Theorem~\ref{t:mainql} immediately implies that we may assume that
$\lam_f$ has some non-degenerate leaves.
Lemma~\ref{l:nodtypes} allows us to not consider some polynomials.

\begin{lem}[Corollary 4.2 \cite{BOPT-QL}]\label{l:nodtypes}
Suppose that one of the following holds for $f$:

\begin{enumerate}

\item $f$ has a fixed parabolic point at which two cycles of
parabolic Fatou domains are attached;

\item $f$ has a locally connected Julia set and an invariant
Siegel domain with two critical points on its boundary.

\end{enumerate}

Then $[f]\in \ol{\phd}_3$.
\end{lem}



\subsection{The description of cubioidal
laminations \cite{BOPT}}\label{ss:descub}

A \emph{stand-alone quadratic invariant gap $U$} is a quadratic
invariant gap $U$ of some lamination considered by itself
(without the lamination). \emph{Quadratic} means that
$\si_3$ on $U'=U\cap \uc$ is two-to-one (except
when $U$ has a \emph{critical} edge $\ol c$ and the point-image of
$\ol c$ has three preimages in $U'$; \emph{critical} means that the
endpoints $\ol c$ map to one point). When studying such a gap $U$,
an important role is played by the map $\psi_U:\bd(U)\to\uc$
collapsing all edges of $U$ and semi-conjugating the map
${\si_3}|_U$ with $\si_2$.

The gap $U$ has a unique edge $M$ such that the arc $H(M)$ (called
the \emph{major hole of $U=U_M$}), which is the component of $\uc\sm
M$ containing no points of $U'$, is of length at least $\frac13$ and
at most $\frac12$. Then $M$ is called the \emph{major $($leaf$)$} of
$U$ and all other edges of $U=U_M$ map to $M$ in finitely many
steps. The set $U'$ consists of all points of $\uc$ which never
enter $H(M)$. In particular, $M$ never enters $H(M)$.


There are two types of majors of $U$. First, $U$ can be of
\emph{regular critical type}. Then $U$ has a \emph{critical} major
$M=\ol{\ta_1\ta_2}$. If a gap $U$ is of regular
critical type then there exists a unique lamination such that $U$ is
its gap. Basically, this lamination is obtained by taking pullbacks of $U$.
This lamination is called the \emph{canonical} lamination of the gap
$U$.

Also, $U$ can be of \emph{periodic type} with its major
$M=\ol{\ta_1\ta_2}$ being a periodic edge of $U$ of period $k$. Call
such $M$ a \emph{major $($leaf$)$ of periodic type}. Clearly,
$\si_3|_{H(M)}$ wraps around the circle while for every $i, 1\le
i\le k-1$ the map $\si_3$ is one-to-one on the circle arc with
endpoints $\si^{\circ i}(\ta_1)$, $\si^{\circ i}(\ta_2)$ disjoint
from $U$. Inside $H(M)$ there are points $\al$, $\be$ such that
$N=\ol{\al\be}$ has the same image as $M$. We call $N$ a
\emph{sibling leaf} of $M$ and construct a gap $V_M$ with edges $M$
and $N$ consisting of all $x$ such that for every $n\ge 0$, the point
$\si^{\circ n}_3(x)$ belongs to the closure of the arc with the same
endpoints as $\si^{\circ n}_3(M)$ disjoint from $U$. Then $V_M$ is a
quadratic gap of period $k$ (i.e., $\si^{\circ k}_3$ maps $V'_M$ onto itself
in a two-to-one fashion) called the \emph{vassal} gap of $U$. If
$M\ne \ol{0\frac12}$ then either $0\in H(M)$ is the only fixed angle
in $H(M)$ and the only angle which stays in $H(M)$ forever, or the
same holds for $\frac12$. By \cite{BOPT} there is a unique
lamination which has both $U$ and $V_M$ as its gaps. Similar to the
regular critical case, this lamination is obtained by taking pullbacks
of $U$ and $V_M$.
It is called the \emph{canonical}
lamination of the gap $U$. Regardless of the type of the gap $U$,
the corresponding canonical lamination is denoted by $\sim_U$ and
the corresponding geo-lamination is denoted by $\lam_U$.

\begin{dfn}\label{d:tunegap} A lamination $\lam$ \emph{tunes a
stand alone quadratic gap $U$ according to a quadratic lamination
$\lam_2$} if all edges of $U$ are leaves of $\lam$, and the map
$\psi_U$ transports the leaves of $\lam$ in $U$ to the leaves of
$\lam_2$. If $\lam$ and $\lam_1$ are two laminations and
$\lam\supset \lam_1$, say that $\lam$ \emph{tunes} $\lam_1$.
\end{dfn}


\begin{dfn}\label{d:co-exist}
A lamination $\lam$ {\em coexists with} a stand-alone quadratic
invariant gap $U$ if every leaf of $\lam$ which intersects an edge
$\ell$ of $U$ in $\disk$ coincides with $\ell$. If the map $\psi_U$
transports the leaves of $\lam$ in $U$ to leaves of a quadratic
invariant lamination $\lam_2$, we say that $\lam$ \emph{weakly tunes
$U$ according to $\lam_2$}. A lamination $\lam$ {\em coexists} with
a lamination $\lam_1$ if no leaf of $\lam$ intersects a leaf of
$\lam_1$ in $\disk$ unless the two leaves coincide.
\end{dfn}

Theorem~\ref{t:cormin-spec} is the main result of \cite{BOPT}.

\begin{thm}
\label{t:cormin-spec} Let $\sim$ be a non-empty cubioidal
lamination. Then {\rm (1) or (2)} occurs $($below $U$ is an
invariant quadratic gap$)$.
\begin{enumerate}

\item The lamination $\sim$ tunes the canonical lamination $\sim_U$
according to a quadratic lamination $\asymp$ from the Main Cardioid,
and if $U$ is of periodic type then the vassal gap $V(U)$ is a gap of
$\sim$.

\item The lamination $\sim$ coexists with the canonical lamination $\sim_U$
and weakly tunes $\sim_U$ on $U$ according to a quadratic lamination
$\asymp$ from the Main Cardioid so that edges of $U$ are not leaves
of $\sim$. Moreover, $U$ is of regular critical type.
\end{enumerate}
\end{thm}


\subsection{The proof of the second part of Theorem B}\label{ss:b2}

Consider a cubic polynomial $f\in\Fc_{nr}$ with $[f]\in \cu\sm \ol\phd^e_3$. By
Theorem C, the lamination $\lam_f=\lam_{\sim_f}$ is cubioidal.
We may assume that all non-repelling periodic points of $f$ but perhaps $0$
have multiplier 1.

Let us define a special set $Q_f\subset K^*$ for the polynomial $f$.
If $0$ is parabolic, by Lemma~\ref{l:nodtypes}
there is a unique $f$-cycle $Q_f\subset K^*$ of Fatou domains at
$0$. By Theorem~\ref{t:mainql}, the map $f|_{J^*}$ is hybrid equivalent to
the appropriate quadratic polynomial $g$ with parabolic fixed point
$a_g$ in its quadratic Julia set $J(g)$. Under this conjugacy $Q_f$
maps to the $g$-cycle $Q_g$ of Fatou domains at $a_g$. Otherwise $0$
is a Siegel point (as Julia sets with Cremer points are
not locally connected). Then the closure $Q_f$ of the Siegel domain
at $0$ is contained in $K^*$.
Below, we consider the two cases of Theorem \ref{t:cormin-spec}.
We will assume that $U^*$ is sufficiently close to $K^*$ which
can be arranged by passing to a suitable pullback.

\emph{Case (1) of Theorem~\ref{t:cormin-spec}}. Denote by $U$ the
quadratic invariant gap tuned by $\lam_f$. Let $M$ be the major of
$U$ and $t\in J(f)$ be the point corresponding to $M$. Then $t$ is
a critical point (in the regular critical case) or a periodic
cutpoint of $J(f)$ (in the periodic type case). As points of neither
type belong to $J(g)$, we have $t\notin K^*$. Take the union
$C$ of both external rays of
$f$ landing at $t$. Let $T$ be the
component of $\C\sm C$ containing $0$. Since $K^*$ is positively
distant from $C$, we may assume that $\ol{U^*}\subset T$. Choose
$x\in J^*$ and all its backward orbits inside $J^*$.
By the construction it is equivalent to considering all
backward orbits of $x$ contained in $T$. The union $B$ of these
backwards orbits is dense in $J^*$. On the other hand, by
\cite{BOPT} the set of arguments of external rays landing at points
of $B$ is dense in the basis $U'$ of the gap $U$. These arguments
will accumulate upon the arguments of external rays landing at $t$
which implies that their landing points converge to $t$, and hence
that $t\in J^*$, a contradiction.

\emph{Case (2) of Theorem~\ref{t:cormin-spec}}. In this case the
invariant quadratic gap $U$ weakly tuned by $\lam_f$ is of
regular critical type, the major $M$ of $U$ does not
belong to $\lam_f$ and hence is contained in a critical gap $G$ of
$\lam_f$. Consider several cases depending on $G$.

\emph{Case A: the gap $G$ is a periodic infinite gap}. 
Consider the quadratic lamination $\asymp$ from the Main
Cardioid according to which $\lam_f$ weakly tunes $U$.
First, assume that $\asymp$ has an invariant Siegel gap $T$ with a
critical edge corresponding to a critical leaf of $\lam_f$ which is
the second critical set of $\lam_f$. Then $G$ corresponds to $T$ and
is an invariant gap with a critical edge on which $\si_3$ is
two-to-one. Hence $\lam_f$ tunes $G$ and fits into case (1) of
Theorem~\ref{t:cormin-spec} considered above.

Second, assume that $\asymp$ corresponds to the lamination
of a polynomial $g$ from the Main Cardioid which has a parabolic
fixed point $a_g$. Then $\asymp$ has an invariant finite gap which
must correspond to a finite invariant gap $H$ of $\lam_f$. The
(pre)periodic edges of $G$ contained in $U$ will be eventually
mapped to edges of $H$; as $G$ is periodic, these leaves will also
remain edges of $G$. Thus, 
$G$ ``rotates'' around $H$.

Suppose that $G$ is quadratic and consider the major $M\subset G$ of
$U$. The properties of quadratic maps imply that under the
appropriate power of $\si_3$ the image of $M$ will be separated from
$H$ by $M$ itself. This contradicts the properties of regular
critical leaves according to which $M$ never enters $H(M)$. Hence
$G$ is either cubic or is of degree 4 (in the latter case there must
exist another critical Fatou gap in the orbit of $G$). However in
both these cases it is easy to see that we have a contradiction with
Theorem~\ref{t:mainql}. Indeed, by Theorem~\ref{t:mainql} the fixed
point $0$ of $f$ corresponds to $a_g$, and hence the orbit of $G$
corresponds to the orbit of a Fatou domain of $g$ at $a_g$ which
must be quadratic.

\emph{Case B: the gap $G$ is a preperiodic infinite gap}.
There is a Fatou domain $\wt G$ of $f$ corresponding to $G$. Note
that $\wt G$ eventually maps to a periodic domain in $K^*$. Denote
by $t\in \bd(\wt G)$ the cutpoint of $J(f)$ which separates $\wt G$
and $Q_f$. If $t\notin K^*$ then we obtain a contradiction as in the
part of this proof corresponding to Case 1 of
Theorem~\ref{t:cormin-spec}. Otherwise
$t$ together with a small arc $I\subset \bd(\wt
G)$ around it eventually maps to $Q_f\subset K^*$. By
the definition of a polynomial-like Julia set $I\subset J^*$. On
one side of $I$ (namely, inside $\wt G$) there are no points of
$J(f)$, hence there are no points of $J^*$ on that side either.
Well-known properties of locally connected quadratic Julia sets
imply that then $I$ must be an arc from the boundary of a Fatou
domain of $J^*$. However the entire $\wt G$ cannot be contained in
$K^*$ because $K^*$ is quadratic-like, a contradiction.

\emph{Case C: the gap $G$ is a preperiodic finite gap}. Then $G$
corresponds to a critical point $t$ of $f$ whose image is still a
cutpoint of $J(f)$. Clearly, $t\notin J^*$ because $J^*$ is
quadratic-like. We may assume that $t\notin \ol{U^*}$. Then we
obtain a contradiction, repeating the arguments from the
part of this proof corresponding to Case 1 of
Theorem~\ref{t:cormin-spec}. \hfill\qed


\begin{thebibliography}{9999}

\bibitem[Bea00]{Beardon}
A.F. Beardon, ``Iteration of rational functions: Complex Analytic
Dynamical Systems'', Springer 2000

\bibitem[BCLOS10]{bclos}
A. Blokh, D. Childers, G. Levin, L. Oversteegen, D. Schleicher,
\emph{An Extended Fatou-Shishikura inequality and wandering branch
continua for polynomials}, arXiv:1001.0953 (2010)

\bibitem[BCO11]{bco11} A. Blokh, C. Curry, L. Oversteegen,
\emph{Locally connected models for Julia sets}, Advances in
Mathematics \textbf{226} (2011), 1621--1661

\bibitem[BFMOT12]{bfmot12} A. Blokh, R. Fokkink, J. Mayer, L. Oversteegen, E.
Tymchatyn, \emph{Fixed point theorems for plane continua with
applications}, Memoirs of the American Mathematical Society,
\textbf{224} (2013), no. 1053

\bibitem[BL02]{bl02} A. Blokh, G. Levin,
\emph{Growing trees, laminations and the dynamics on the Julia set},
Ergodic Theory and Dynamical Systems \textbf{22} (2002), 63--97

\bibitem[BO06]{bo06} A. Blokh, L. Oversteegen, \emph{The Julia sets
of quadratic Cremer polynomials}, Topology and its Appl.,
\textbf{153} (2006), 3038--3050.

\bibitem[BOPT13a]{BOPT}
A. Blokh, L. Oversteegen, R. Ptacek, V. Timorin, \emph{Laminations
from the Main Cubioid}, preprint (2013)

\bibitem[BOPT13b]{BOPT-QL}
A. Blokh, L. Oversteegen, R. Ptacek, V. Timorin,
\emph{Quadratic-like dynamics of cubic polynomials}, preprint (2013)


\bibitem[BuHe01]{BH}
X. Buff, C. Henriksen, \emph{Julia Sets in Parameter Spaces},
Commun. Math. Phys. \textbf{220} (2001), 333 -- 375



\bibitem[DH8485]{DH}
A. Douady and J. Hubbard, \emph{\'Etude dynamique des polyn\^omes
complex I \& II}, Publ. Math. Orsay (1984--85)


\bibitem[Fat20]{fat20} P. Fatou, \emph{Sur les equations functionnelles}, Bull. Soc.
Mat. France \textbf{48} (1920)

\bibitem[GM93]{GM}
L.R. Goldberg, J. Milnor, \emph{Fixed points of polynomial maps.
Part II. Fixed point portraits}, Ann. scientifiques de l'\'E.N.S.
$4^e$ s\'er., \textbf{26}, No. 1 (1993), 51--98


\bibitem[Hub93]{Hu}
J. H. Hubbard, \emph{Local connectivity of Julia sets and
bifurcation loci: three theorems of Yoccoz}, in: Topological Methods
in Modern Mathematics, Publish or Perish (1993)

\bibitem[Kiw00]{Ki}
J. Kiwi, \emph{Non-accessible critical points of Cremer
polynomials'', Ergodic Theory and Dynamical Systems}, \textbf{20}
(2000), Issue 05, 1391--1403

\bibitem[Kiw02]{kiw02} J. Kiwi, \emph{Wandering orbit portraits},
Trans. Amer. Math. Soc. \textbf{354} (2002), 1473-–1485

\bibitem[Kiw04]{kiwi97} J. Kiwi, \emph{$\mathbb R$eal laminations and the
topological dynamics of complex polynomials}, Advances in
Mathematics, \textbf{184} (2004), 207--267

\bibitem[KvS06]{KvS} O. Kozlovski, S. van Strien, \emph{Local connectivity
and quasi-conformal rigidity of non-renormalizable polynomials},
Proc. of the LMS, \textbf{99} (2009), 275--296

\bibitem[MSS83]{MSS}
R. Ma\~n\'e, P. Sad, D. Sullivan, \emph{On the dynamics of rational
maps}, Ann. Sci. \'Ecole Norm. Sup. (4) \textbf{16} (1983), no. 2,
193--217.



\bibitem[Mil00b]{M03}
J. Milnor, \emph{Local connectivity of Julia sets: expository
lectures}, In: The Mandelbrot Set, Theme and Variations, Cambr.
Univ. Press (2000), 67--116

\bibitem[Mil06]{M}
J. Milnor, ``Holomorphic dynamics'', 3rd Ed., Matrix Press, 2006

\bibitem[Shi87]{shi87} M. Shishikura, \emph{On the quasiconformal surgery of
rational functions}, Ann. Sci. Ecole Norm. Sup., \textbf{20} (1987),
1--29

\bibitem[Thu85]{thu85} W.~Thurston. \newblock {\em The combinatorics of
iterated rational maps} (1985), in: ``Complex dynamics: Families and
Friends'', ed. by D. Schleicher, A K Peters (2008), 1--108


\end{thebibliography}
\end{document}